\documentclass[11pt]{amsart}
\usepackage{amsmath, amssymb, mathrsfs}
\usepackage{amssymb}
\usepackage{mathpazo}
\usepackage{extarrows}
\usepackage[T1]{fontenc}
\def\lbr{\left\{}
\def\rbr{\right\}}

\def\mso{{\mathscr O}}

\def\me{{\mathbb  E}}

\def\mr{{\mathbb  R}}

\def\mp{{\mathbb  P}}
\def\mi{{\mathbb I}}

\DeclareMathOperator{\str}{str}
\DeclareMathOperator{\Tr}{Tr}
\DeclareMathOperator{\Pf}{Pf}

\newcommand{\rf}[1]{(\ref{#1})}
\def\lnfrac#1#2{\raise.7ex \hbox{\Small $#1$}
\kern-.15em/\kern-.15em  \lower.2ex \hbox{\Small $#2$}}

\theoremstyle{plain}
\newtheorem{theorem}{Theorem}[section]

\newtheorem{lemma}[theorem]{Lemma}

\theoremstyle{definition}

\newtheorem{remark}[theorem]{Remark}

\numberwithin{equation}{section}

\begin{document}

\title[Reflecting Brownian Motion and the Gauss-Bonnet-Chern Theorem]{Reflecting Brownian Motion and the Gauss-Bonnet-Chern Theorem}
\vskip 1cm

\author{Weitao Du}
\address{School of Mathematics\\
University of Science and Technology of China\\
Hefei, Anhui Province, China
}

\author{Elton P. Hsu}
\address{Department of Mathematics\\
        Northwestern University \\ Evanston, IL 60208, USA}
\email{ehsu\@@math.northwestern.edu}

\email{weitao.du\@@ustc.edu.cn}



\begin{abstract} We use reflecting Brownian motion (RBM) to prove the well known Gauss-Bonnet-Chern theorem for a compact Riemannian manifold with boundary. The boundary integrand is obtained by carefully analyzing the asymptotic behavior of the boundary local time of RBM for small times. 
\end{abstract}
\maketitle

\section{Introduction and outline of the approach}
Let $M$ be a smooth oriented manifold with boundary.  Its Euler characteristic $\chi(M)$ is defined by 
$$\chi (M) = \sum_{i=0}^{m} (-1)^i b_i,$$
where $b_i = \text{dim}\, H^i(M)$ are the Betti numbers.  These numbers as well as the Euler characteristic $\chi(M)$ are topological invariants of the smooth 
manifold $M$.   Assume that we equip $M$ with a Riemannian metric. A fundamental result in differential geometry is the Gauss-Bonnet-Chern theorem (Chern~\cite{Chern}), which expresses 
the Euler characteristic as the sum of two integrals on $M$ and on its boundary $\partial M$
$$\chi(M) = \int_M e_M(x)\, dx + \int_{\partial M} e_{\partial M}(\overline x)\, d\overline x, $$
where $e_M$ and $e_{\partial M}$ are local gometric invariants determined by the curvature tensor and the second fundamental form of the boundary of the 
manifold.  The interior integrand can be described easily. When the dimension of the manifold $d$ is even, the interior integrand $e_M = \Pf (\Omega)/(4\pi)^{d/2} (d/2)!$, which is the density of the Euler class of the manifold. When $d$ is odd,  we have $e_M = 0$. 

By the de Rham theorem (see Duff~\cite{Duff}),  the topological cohomology group $H^*(M)$ can be identified with the de Rham cohomology group on the bundle of differential forms satisfying the absolute 
boundary condition.  For a manifold without boundary, McKean and Singer~\cite{McKeanS} expressed the integral of the Euler form on the manifold in terms of the supertrace of the heat kernel of the 
Kodaira-de Rham Laplacian  on the bundle of differential forms 
\begin{equation}
\chi(M) = \int_M \str p^*(t, x,x)\, dx.\label{mckeansinger}
\end{equation}
They conjectued a ``fantastic cancellation'',  namely the supertrace of the heat kernel has a pointwise limit as $t\downarrow0$.  This conjecture was proved by Patodi~\cite{Patodi}.
Getzler~\cite{Getzler} gave a simple analytic proof of the general Atiyah-Singer index theorem.  For details of this approach to the Gauss-Bonnet-Chern theorem see 
Berline et al ~\cite{BerlineGV}  and Yu~\cite{Yu}.  Another independent approach was developed by 
Gilkey~\cite{Gilkey} \cite{Gilkey1}  based on his geometric invarance theory, in which the integrands on $M$ as well as the boundary $\partial M$ are identified by the invariance properties 
they must possess. 

The probabilistic approach to analytic index theorems was initiated by Bismut~\cite{Bismut} for Riemannian manifolds without boundary. The crucial observation of this approach is the heat 
kernel on differentiable forms associated with the Hodge-de Rham Laplacian can be expressed explicitly in terms of a multiplicative Feynman-Kac functional of the horizontal 
Brownian motion on the frame bundle of the Riemannian 
manifold $M$. It then becomes a purely algebraic fact that the ``fantastic cancellation'' occurs for the supertrace of this multiplicative Feynman-Kac functional at the pathwise level with the 
result that the leading term is precisely the Euler form, thus resulting in a so-called stochastic local index theorem.  Later Bismut's approach was geatly simplified by Hsu~\cite{Hsu}.  
See also Hsu~\cite{Hsu1} for a more detailed exposition of this approach to the local index theorem for the Kodaira-Beltrami Laplacian (Gauss-Bonnet-Chern index theorem) as well as for the 
Dirac operator for spin manifolds.

Bismut's approach and Hsu's simplification cannot be easily carried over to a manifold with boundary.  The major difficulty lies in the fact that the multiplicative functional is discontinuous at the 
boundary.  Another difficulty is the presence of the boundary local time, with the result that the ``fantastic cancellation'' does not occur at the path level; more precisely, it occurs completely in 
the interior of the manifold and does so only partially at the boundary.  Such difficulties should be expected, for in the limit the supertrace str$p_M(t,x,x)$ in \rf{mckeansinger} 
must have a singularity of a different order near the boundary. This singularity will eventually produce the integral on the boundary, which is singular with respect to the Riemannian volume measure.   Using Malliavin calculus Shigekawa et al~\cite{ShigekawaUW} 
successfully handled the singular terms on the boundary, thus giving a probabilistic proof of the Gauss-Bonnet-Chern formula.  Malliavin calculus allows the heat kernel to be regarded as a generalized functional on Brownian motion which  can then expanded as an asymptotic power series in the small time parameter in a generalized sense.  This probabilistic approach cannot be regarded as a 
direct extension of the Bismut-Hsu approach because of the use of Malliavin calculus. Over the years there have been several unsuccessful attempts at overcoming the difficulties mentioned
above in order to give a simplified probabilistic approach to the Gauss-Bonnet-Chern theorem for manifolds with boundary. The purpose of the present work is to give an elementary 
probabilistic proof of the  Gauss-Bonnet-Chern theorem without using Malliavin calculus.  Our approach is heavily indebted to the previous works by Hsu~\cite{Hsu} and 
Shigekawa et al.~\cite{ShigekawaUW}. In order to help the reader understand our approach, in the remainder of this section we outline the major steps of the paper.

We first review some basic facts about the McKean-Singer analytic formula expressing the Euler characterstic as the integrated supertrace of the Hodge-de Rham heat kernel with the 
absolute boundary condition on the bundle of differential forms.  Then we show how the heat kernel can be represented probabilistically in terms of reflecting Brownian motion
and its lift to the bundle or differential forms via a Feynman-Kac formula using a multiplicative functional.  This multiplicative functional satisfies a special ordinary differential equation which
takes a slightly simpler form in the semi-geodesic coordinates we will use later.  Since the McKean-Singer formula \rf{mckeansinger} is an integral with respect to the Riemannian volume measure, it is 
necessary to separate the supertrace of the heat kernel to single out the part that contributes to the boundary integral.  For this purpose we use the double $\widetilde M$ of the manifold $M$ 
and consider the reflecting Brownian motion on $M$ as the projection of the Brownian motion on $\widetilde M$ to $M$. In this way,  the heat kernel is naturally decomposed into a 
sum of two parts and we recognize that the part on the mirror image of $M$ contributes to the boundary integral of the theorem. 

Next, by the principle of not feeling the boundary, as the time parameter decreases to zero, the calculation of the supertrace of the heat kernel is highly concentrated in an arbitrarily small but 
fixed neighborhood of the base point. For this reason we can without loss of generality choose the semi-geodesic coordinates to carry out our computations of the boundary term locally in 
a neighborhood of the boundary, which can be regarded as a half space. The advantages of 
semi-geodesic coordinates are twofold: first, it singles out the normal component so that the diffusion process in this direction satisfies a relatively simple stochastic differential equation for 
easy asymptotic computation; second, it shows clearly the role of the second fundamental form through the local expansion of the Riemannian metric near the boundary so that the expansion 
of the multiplicative functional  can be easily obtained.   In semi-geodesic coordinates, the expansion of the stochastic parallel transport and the multiplicative functional can be carried 
out near the boundary in integrated powers of the time and the boundary local time, with the boundary local time counting as one half power of the time. After taking the supertrace, 
the leading nonvanishing terms can be identified.  It turns out that there are $(d+1)/2$ or $d/2$ such terms depending on whether the dimension 
of the manifold is odd or even. They are expressed in terms of the curvature tensor and the second fundamental form at the boundary. The dimension dependent universal constants can be obtained
by evaluating integrated pwoers of the local time of a classical Brownian bridge. 

\section{Heat equation on differential forms on manifold with boundary}

At the boundary of a Riemannian manifold, every differential form $\omega$ has a unique orthogonal decomposition of the form
\begin{equation}
\omega = \omega_{\rm tan} + \omega_{\rm norm} = \omega_1 + d\nu \wedge \omega_2\ ,\label{norm}
\end{equation}
where $d\nu$ is the dual 1-form of the inward unit normal vector and $\omega_1$ and $\omega_2$ have only tangential components. 
The absolute boundary condition is 
\begin{equation}
		\omega_{\text{\rm norm}} = 0, \ \ \  (d\omega)_{\text{\rm norm}} = 0.\label{absolute}
\end{equation}
Let $\delta$ be the formal adjoint operator of the exterior differential $d$ with respect to the standard inner product on $\varGamma(T^*M)$, the space of sections of the differential forms. The Hodge-de Rham Laplacian is defined by:
$$\square_M := -(d\delta + \delta d)\ .$$
It has a unique self-adjoint extension satisfying the absolutely boundary condition \rf{absolute}.  By the de Rham theorem, the space
$$H^k_{\text{abs}} (M;\mathbb{C}) : = \lbr\omega \in\varGamma(\wedge^kM): \square_M \omega = 0, \ \omega_{\rm norm} = 0\rbr$$
can be canonically identified with $H^k (M;\mathbb{C})$,  the absolute simplicial cohomology group of $M$. Note that the two conditions in the above definition implies the second condition 
in \rf{absolute}. It follows that the Euler characteristic of the manifold $M$ is given by 
$$\chi(M) = \sum_k (-1)^k\text{dim} \, H_{\text{abs}}^k (M;\mathbb{C}).$$
Let $\omega_0$ be a differential form on $M$ and consider the following initial boundary valued problem for $\omega = \omega(t,x)$:
\begin{equation}
\left\{
\begin{array}{l}
\frac{\partial\omega}{\partial t} = \frac{1}{2} \square_M \omega,\\
\omega (\cdot,0)= \omega_0,\\

\omega_{\rm norm} = 0,\ (d\omega)_{\rm norm}=0.
\end{array}
\right. \label{heat}
\end{equation}
By general theory of parabolic differential equations, the solution of the above problem can be written in the form
$$\omega(t,x) = \int_M p ^*(t,x, y)\omega_0(y)\, dy, $$
where $p^* (t,x,y):  \wedge_y^kM\rightarrow \wedge_x^kM$ is the heat kernel.  Note that $p^* (t,x,x)$ is a homomorphism from $\wedge_x^kM$ into itself. 

Suppose $V$ is a finite dimensional vector space. Each $T\in\text{End}(V^*)$ can be extended to the whole alternating tensor algebra $\wedge^* V$ uniquely as a derivation (still denoted by $T$):
$$T(\theta_1 \wedge \theta_2) = (T \theta_1) \wedge \theta_2 + \theta_1 \wedge (T \theta_2).$$
Thus $T: \wedge^*V\rightarrow \wedge^*V$ is a degree-preserving linear map. 
We define the supertrace $\str(T)$ for a degree-preserving linear map $T: \wedge^* V \rightarrow \wedge^* V$ by
$$ \str(T) = \sum_{k=0}^{n} (-1)^k \Tr (T|_{\wedge^k V})\ .$$
We have the following McKean-Singer formula for the Euler characteristic (see {\sc Theorem} 7.3.1 in Hsu~\cite{Hsu}).
\begin{theorem} For for all $t > 0$ we have
\begin{equation}
\chi(M) = \int_{M} \str p^* (t,x,x) dx. \label{Euler}
\end{equation}
\end{theorem}
We will use this by taking the limit of the supertrace as $t\downarrow0$. 

\section{Local geometry near the boundary}

Let $\partial M$ be the boundary of $M$. Near a point $o\in \partial M$,  the manifold can be parametrized by a system of semi-geodesic coordinates $x = (x^1, \overline{x} )$, 
where $x^1$ is the distance of $x$ to $\partial M$ and $\overline{x} = (x^2, \ldots, x^n)$ a system of normal coordinates of $\partial M$ centered at $o$. In these coordinates, 
the Riemannian metric is given by
\begin{equation} 
g_{ij}(x) = \begin{cases}
\delta_{1j}, &i = 1, 1\le j\le n;\\
\delta_{ij} + 2H_{ij}x_1 +  O(\vert x\vert^2),\ &2\le i, j\le n.
\end{cases}\label{cartan}
\end{equation}
Here
$$H_{ij}=\left\langle \frac\partial{\partial x^1}, \nabla_{\partial/\partial x^i}\frac\partial{\partial x^j}\right\rangle$$
is the second fundamental form of $\partial M$ at $o$. In general, the  local coordinates are only defined in a 
neighborhood of the origin of $\mathbb{R}_{+}^n$. However, we can extend the local chart $(x^1, \overline{x})$ to the whole $\mathbb{R}_{+}^n$ in such a way that
\rf{cartan} holds for $g_{ij}$ on the whole space. By the principle of not feeling the boundary (see Hsu\cite{Hsu1}),  the computation of the boundary term as $t\downarrow 0$ 
(see the second term in \rf{double2} below) can be localized near a fixed point at the boundary. Therefore we may consider the manifold $M = \mr^n_+$ with an Riemannian metric satisfying 
\rf{cartan} globally on $\mr_+^n$.  A crucial simplifying advantage of such a coordinate system is that the normal and tangential projection operators $P$ and $Q =I-P$ are expressed as 
constant matrices.

Next, we examine the normal component of $\widetilde E$ in the semi-geodesic coordinates.  We have $g^{11} = 1$ and $g^{1i} = 0$ for\ $i\ge 2$,  hence the 
Laplace-Beltrami operator in these coordinates has the special form
\begin{equation}
\Delta = \frac{\partial^2}{\partial x^1\partial x^1} + b^1 \frac{\partial}{\partial x^1} + \sum_{i,j=2}^{d} g^{ij} \frac{\partial^2}{\partial x^i \partial x^j} + \sum_{i =2}^{d} b^i \frac{\partial}{\partial x^i}
\label{lb}
\end{equation}
where 
\begin{equation}
b^i (x) = \frac{1}{\sqrt{\det g(x)}} \frac{\partial}{\partial x^j}\left(\sqrt{\det g(x)} g^{ji} (x)\right).\label{driftb}
\end{equation}

\section{Reflecting Brownian motion and the probabilistic representation}

If a differential form on $M$ is decomposed into the tangent and the normal components  on the boundary
$$\omega = \omega_{\rm tan} + \omega_{\rm norm}$$
as described in \rf{norm}, at each point $x\in\partial M$ we define the normal projection $P$ and the tangential projection $Q= I-P$ by 
$$P\omega = \omega_{\rm norm}= d\nu\wedge\omega_2 \ \text{   and   } \ Q\omega = \omega_{\rm tan}.$$

Let $H: T_{x} \partial M \rightarrow T_{x} \partial M,\ x \in \partial M$ be the second fundamental form at $x\in\partial M$.  Its dual acts on $T_x^*\partial M$ and can be extended to 
$\wedge^* \partial M$ as a derivation in the fashion mentioned before. We denote this extension still by $H$. Then the absolute boundary condition is equivalent to
$$Q[\nabla_\nu - H ]\omega- P\omega = 0 \ \ \text{ on } \partial M,$$
where $\nabla_\nu$ is the covariant derivative along  $\nu$; see Hsu~\cite{Hsu3}.
For the probabilistic representation of the heat equation on forms, it is necessary to lift the heat equation to the orthonormal frame bundle $\mso(M)$. 
On the orthonormal frame bundle the projections $P$ and $Q$ have their natural horizontal liftings (still denoted by the same letters). The covariant derivative $\nabla_\nu\omega$ becomes 
the usual derivative of the scalarization $\widetilde\omega$ along the horizontal lift $\widetilde\nu$. Hence the absolute boundary condition becomes
$$Q[\widetilde{\nu} - H ] \widetilde{\omega} - P \widetilde{\omega} = 0 \text{ on } \  \ \partial \mso (M).$$

We now consider the probabilistic representation of the solution of the heat equation \rf{heat}.  Consider the following stochastic differential equation on the frame bundle $\mso(M)$:
$$d U_t = \sum_{i=1}^{n} H_i (U_t) \circ dB_t^i + \tilde{\nu}(U_t)dl_t.$$
Here $\lbr H_i\rbr$ are the canonical horizontal vector fields on $\mso(M)$ and $l$ is the boundary local time for $U$. 
The solution ${U_t}$ is a horizontal reflecting Brownian motion and the projection $X_t = \pi U_t$ is the usual reflecting Brownian motion on $M$ and
$l$ is also the boundary local time of $X$. 

Let $\widetilde\omega = \widetilde\omega(u, t)$ be the scalarization of the solution of \rf{heat}.  We have the following representation
\begin{equation}
\widetilde{\omega}(u,t) = \mathbb{E}_u \{M_t \widetilde{\omega}_0 (u_t)\}. \label{kac}
\end{equation}
Here $\lbr M_t\rbr$ is a right continuous matrix valued multiplicative functional (see Hsu~\cite{Hsu1} and Ouyang~\cite{Ouyang}.  In the semi-geodesic coordinates
the projection operators $P(x)$ and $Q(x)$ are represented by constant matrices independent of $x\in\partial M$, hence we can regard them as
defined on the whole manifold as constant matrices and denote them simply by $P$ and $Q$.  Decomposing $M_t$ into the normal and tangential components as
$$M_t = M_tP + M_tQ\xlongequal{\rm def} Y_t+Z_t.$$
Then $M_t$ is the solution of the following system of equations (see (5.10) of Ouyang~\cite{Ouyang}):
$$\begin{cases}
Y_t = I_{\{t < T_{\partial M}\}} e(0,t)P +  I_{\{t > T_{\partial M}\}} Z_{t_{*}}e(t_{*},t)P,\\
Z_t = Q + \displaystyle{\int_{0}^{t}}(Y_s + Z_s ) d \chi_s ,
\end{cases}$$
where $t_{*} =\sup \{ s \leq t : x_s \in \partial M \}$ is the last exit time of the boundary before $t$,  
$$d\chi_s = -H(u_s)dl_s + \frac{1}{2}\Omega( u_s)ds,$$
and $\{ e(s,t), t \ge s \}$ is  the solution of
$$\frac{d}{dt} e(s,t) =\frac{1}{2} e(s,t) \Omega( u_t), \ \ \ \ e(s,s)= I.$$
Here $\Omega$ is the canonical action of the curvature tensor on the bundle of differential forms regarded as a linear combination of twice compositions  of $\text{End}(\mr^d)$; namely, 
$$\Omega \in (\mathbb{R}^d)^* \otimes \mathbb{R}^d \otimes (\mathbb{R}^d)^* \otimes \mathbb{R}^d = \text{End}(\mathbb{R}^d) \otimes_\mr\text{End}(\mathbb{R}^d).$$
It is denoted by $\Omega$ in Hsu~\cite{Hsu},  page 195.
Note that the $dQ$ term in Ouyang~\cite{Ouyang} vanishes in the semi-geodesic coordinates because $Q$ is a constant.

Before Brownian motion hits the boundary (i.e., $t \leq  T_{\partial M}$), the process $M_t$ evolves simply by $dM_t = (1/2)M_s\Omega( u_t)dt$, or equivalently, 
\begin{equation}
M_t= I + \frac{1}{2} \int_{0}^{t} M_s \Omega( u_s) ds. \label{before}
\end{equation}
So at the first hitting time $T_{\partial M}$ of the boundary,  the tangent component is 
$$M_{ T_{\partial M}} Q = e(0,  T_{\partial M}) Q.$$
Although $M_t$ is discontinuous at this time, $Z_t = M_t Q$ is continuous and satisfies the differential equation $dZ_t = (Y_t + Z_t ) d \chi_t$ after the hitting time with the initial condition $Z_{ T_{\partial M}} 
= e(0, T_{\partial M})Q$. Substituting the equation of $Y_t$ into this equation, we obtain a self-contained equation for $Z_t$ in the semi-geodesic coordinates for $t> T_{\partial M}$,
\begin{equation}
Z_t = Z_{T_{\partial M}} + \int_{T_{\partial M}}^{t} \Phi (Z_s)d\chi_s, 
\end{equation}
where $\Phi(Z_s)$ is given by
\begin{equation}
\Phi(Z_s) = Z_{s_*}e(s_*,s) P + Z_s.
\end{equation}
The normal component $Y_t = M_t P$ starts from zero immediately after each time the path hits the boundary.  Thus for $t > T_{\partial M}$ ,
$$Y_t = Z_{t^*} e(t^*, t) P= (Y_{t^*} + Z_{t^*}) e(t^* , t) P=  M_{t^*} e(t^* , t) P.$$
Therefore for  $t > T_{\partial M}$,
\begin{equation}
M_t = Y_t + Z_t = Q + \int_{0}^{t} M_s d\chi_s + M_{t^*} e(t^* , t) P. \label{after}
\end{equation}

\section{Doubling of the manifold}

For asymptotic calculations involving the Neumann boundary condition, it is more convenient to consider the doubling of the manifold so that we work on a manifold without boundary. 
Denote the doubled manifold along the boundary $\partial M$ by $\widetilde{M} = M \cup M^*$. For each $x \in M$, there is a mirror point on $M^*$ and we denote it by $x^*$.   
We denote the heat kernel on $\widetilde{M}$ by $q(t,x,y)$.  The Neumann heat kernel $p_0(t,x,y)$ on functions on $M$ can be expressed in terms of the heat kernel $q(t,x,y)$ on $\widetilde M$ by 
\begin{equation}p_0 (t, x, y) = q(t, x, y) + q(t, x , y^*).\  \label{kernel}
\end{equation}
Denote the Riemannian Brownian motion on $\widetilde M$ by $E_t$, then the projection $X_t = \pi(E_t)$ is a reflecting Brownian motion on $M$. 

Now we are ready to express the heat kernel $p^* (t,x,y)$, the heat kernel on differential forms,  in terms of quantities on the doubled manifold $\widetilde{M}$.  Suppose $F$ is a function defined on the 
path space of $M$. We can extend it to a function defined on paths of $\tilde{M}$ by $\widetilde F(x) = F (\pi \circ x)$.  From (\ref{kac})  we can write
\begin{align*}
&\quad \mathbb{E}_x [M_t u_t^{-1}\omega_0(X_t) ]\\
&= \mathbb{E}_x \left[\mathbb{E}_{t;x,E_t}\{M_t u_t^{-1}\}\omega_0(E_t)\right]\\ 
& = \int_{M} q(t,x,y)\mathbb{E}_{t;x,y}\{M_t u_t^{-1}\}\omega_0(y) dy+ \int_{M^*} q(t,x,y)\mathbb{E}_{t;x,y}\{M_t u_t^{-1}\}\omega_0(y) dy\\  
& = \int_{M} [q(t,x,y)\mathbb{E}_{t;x,y}\{M_t u_t^{-1}\} + q(t,x,y^*)\mathbb{E}_{t;x,y^*}\{M_t u_t^{-1}\}]\omega_0(y) dy.
\end{align*}
Here $\me_{t;x,y}$ denotes the expectation with respect to Brownian bridge from $x$ to $y$ in time $t$ on the double $\widetilde M$. 
Comparing the solution of \rf{heat} with the left hand side of (\ref{kac}), we have
\begin{equation}
p^* (t,x,x) = q(t,x,x)\mathbb{E}_{t;x,x}\{M_t u_t^{-1}\} + q(t,x,x^*)\mathbb{E}_{t;x,x^*}\{M_t u_t^{-1}\}. \label{double}
\end{equation}
The multipicative functional $M$ behaves differently depending on whether the path hits the boundary or not, so we separate the two cases by 
$$\mathbb{E}_{t;x,x}\{M_t u_t^{-1}\} = \mathbb{E}_{t;x,x}\{M_t u_t^{-1}\mi_{\lbr T_{\partial M}\le t\rbr} \} + \mathbb{E}_{t;x,x}\{M_t u_t^{-1} \mathbb{I}_{\lbr T_{\partial M} > t\rbr} \}.$$

Let $W(\tilde{M})$ be the path space on $\tilde{M}$. We define the reflection map $R: W(\tilde{M}) \rightarrow W(\tilde{M})$ as follows:
$$(Rx)_s = \left\{
\begin{array}{lr}
	 x_s, \ \ \ \  &  s < T_{\partial M} ,  \\
	 x_s^*, \ \ \  & s \ge T_{\partial M}.
\end{array}
\right.$$
Namely, we reflect the path to its mirror point after it hits the boundary for the first time. The next lemma builds a connection between the Brownian bridge from $x$ to $x$ and from $x$ to $x^*$.

\begin{lemma} \label{reflection} Suppose that $G$ is a function on the path space $W(\widetilde{M})$. For $x\in M$, 
	$$	
	\mathbb{E}_{t;x,x} \{  G(R E )\mathbb{I}_{\lbr T_{\partial M} \leq t\rbr}\} = \frac{q(t,x,x^* )}{q(t,x,x )} \mathbb{E}_{t;x,x^*} \{G(E)\}\ .
$$
\end{lemma}
\begin{proof} It is clear  that
	$ T_{\partial M} (RE) = T_{\partial M}(E)$ and  $(RE)^*_t = E_t$ for $t\ge T_{\partial M}$. Suppose that $f\in C^{\infty} (\widetilde M)$. We have
\begin{align}
	&\quad\int_{\widetilde M} f(y)q(t,x,y) \mathbb{E}_{t,x,y} \{G(RE)\mathbb{I}_{T_{\partial M \leq t}}  \} dy \nonumber \\
	&= \mathbb{E}_x \{f(E_t) G(RE)\mi_{\lbr T_{\partial M}\le t\rbr} \} \nonumber\\
	& =\mathbb{E}_x \{f((RE)_t^*) G(RE)\mi_{\lbr T_{\partial M}\le t\rbr} \} \nonumber\\
	& =\mathbb{E}_x \{f(E_t^*) G(E)\mi_{\lbr T_{\partial M}\le t\rbr} \}\ .  \nonumber
\end{align}
The last line follows from the standard reflection principle stating that $(RE)_t$ is also a Brownian motion on $\tilde{M}$  starting from $x$. On the other hand,
	$$\mathbb{E}_x \{f(E_t^*) G(E)\mi_{\lbr T_{\partial M}\le t\rbr}\} = \int_{\widetilde M} f(y)q(t,x,y^*) \mathbb{E}_{t,x,y^*} \{ G(E)\mi_{\lbr T_{\partial M}\le t\rbr}\} dy\ .$$ 
Hence the equality
\begin{align*}
\int_{\widetilde M} f(y)q(t,x,y)& \mathbb{E}_{t,x,y} \{G(RE)\mi_{\lbr T_{\partial M}\le t\rbr} \} dy \\
= & \int_{\widetilde M}f(y)q(t,x,y^*) \mathbb{E}_{t,x,y^*} \{ G(E)\mi_{\lbr T_{\partial M}\le t\rbr} \} dy
\end{align*}
holds for every test function $f\in C(\widetilde M)$, which implies that
$$q(t,x,y) \mathbb{E}_{t,x,y} \{G(RE)\mi_{\lbr T_{\partial M}\le t\rbr}\} = q(t,x,y^*) \mathbb{E}_{t,x,y^*} \{\mi_{\lbr T_{\partial M}\le t\rbr}G(E) \}\ .$$
The lemma follows by letting $y = x$ and using $\mp_{t, x, x^*}\lbr T_{\partial M}\le t\rbr = 1$.
\end{proof}

We are mainly concerned with $M_t$ and $u_t^{-1}$ as functions on the path space $W(M)$ and $W(\widetilde M)$. They are related by 
$$M(RE)_t = M(E)_t \quad\text{and}\quad u_t^{-1}(RE) =  u_t^{-1}(E).$$
Applying the preceding lemma to $M_tu_t^{-1}$ we have
$$\mathbb{E}_{ t;x,x} \{M_t u_t^{-1} \mathbb{I}_{T_{\partial M} \leq t} \} = \frac{q(t,x,x^*)}{q(t,x,x)} \mathbb{E}_{t;x,x^*} \{ M_t u_t^{-1} \}\ .$$
Combining this with \rf{double}, we conclude that
\begin{align}
&p^* (t,x,x) \label{double2}\\
= & \ q(t,x,x)\mathbb{E}_{t;x,x}\{M_t u_t^{-1}\mi_{\lbr T_{\partial M}> t\rbr}\}+ 2q(t,x,x^*)\mathbb{E}_{t;x,x^*}\{M_t u_t^{-1}\}.\nonumber
\end{align}
On the set $\lbr T_{\partial M}> t\rbr$ the boundarly local time is zero.  For all $x\in M\backslash\partial M$, we have 
$\mp_{t;x, x}\lbr T_{\partial M}> t\rbr\uparrow 1$ and $(2\pi t)^{d/2}q(t,x,x) \rightarrow1$ as $t\downarrow 0$. It follows that as in the case without boundary, 
$$\lim_{t\downarrow0}q(t,x,x)\me_{t;x,x}\left[\str\lbr M_t u_t^{-1}\rbr\mi_{\lbr T_{\partial M}> t\rbr}\right]\rightarrow e_M(x)$$
and 
$$\lim_{t\downarrow0}\int_{M}q(t,x,x)\me_{t,x,x}\left[\str\lbr M_tu_t^{-1}\rbr\mi_{\lbr T_{\partial M}> t\rbr}\right] dx = \int_M e_M(x)\, dx.$$
This yields the integral on the manifold in the Gauss-Bonnet-Chern formula.  We will show below that it is exactly the second 
term in \rf{double2} that contributes the boundary integral in the Gauss-Bonnet-Chern formula.

\section{Asymptotic calculations near the boundary}

On the double $\widetilde M$ we have $x^* = (-x^1 , \overline{x})$ and $\pi (x^1, \overline x) = (\vert x^1\vert, \overline x)$.
The expansion of the Riemannian metric in \rf{cartan} holds with $x^1$ there replaced by $\vert x^1\vert$. Thus the Riemannian 
metric is smooth on each copy of the manifold but its derivative along the normal direction has a jump described precisely by the second fundamental form $H$ and 
the Laplace-Beltrami operator $\Delta_{\widetilde M}$ on $\widetilde{M}$ has jumps in the coefficients of the first derivatives accordingly. More precisely, 
the drift $b^1 (x^1,\overline{x})$ in the Laplace-Beltrami operator on $\widetilde M$ is discontinuous at the boundary and satisfies $b^1 (x^*) = - b^1 (x)$.
We have seen that in the Laplace-Beltrami operator \rf{lb} the derivatives with respect to $x^1$ are separated from those in the other coordinates.  Thus the normal 
component $E_t^1$ satisfies the stochastic differential equation of the form
\begin{equation}
d E_s^1 = d W_s^1 + b^1 (E_s)ds,
\end{equation}
where $W^1$ is a standard one-dimensional Brownian motion.  

We first analyze the behavior of the integration of \rf{Euler} near the boundary. Fix a positive small $h$ and  consider the collar $\partial M \times [0,h]$ near the boundary.  In the semi-geodesic coordinates 
we have
\begin{equation}
\frac{\sqrt{\det g(x^1 , \overline{x})}}{\sqrt{\det g(0 , \overline{x})}} = 1 + O(|x^1 |), \label{volcomp}
\end{equation}
for $(x^1 , \overline{x}) \in \partial M \times [0,h]$. This shows that asymptotically the Riemannian volume measures on $M$ and on $\partial M$ are related by $\text{vol}_M = \text{vol}_{\partial M}\cdot dx^1$. 
Since the lines $\overline x = $ constant are geodesics, by (4.9) of Hsu~\cite{Hsu2},
\begin{align*}
q(t,x,x^*) = & \left[\frac{1}{(2 \pi t )^{d/2}} + O\left(\frac{1}{t^{(d-1)/2}}\right)\right]e^{- d(x,x^*)^2 /2t} \\
	=& \left[\frac{1}{(2 \pi t )^{d/2}} + O\left(\frac{1}{t^{(d-1)/2}}\right)\right]e^{- 2\vert x^1\vert^2 /t}.
\end{align*}
Now we have
\begin{align}\label{integration1}
& \int_{\partial M \times [0, h]} q(t,x,x^*)\mathbb{E}_{t;x,x^*}\str \{M_t u_t^{-1}\}\text{vol}_M(dx) \\ 
\sim& \int_{\partial M \times [0, h]}q(t,x,x^*) \mathbb{E}_{t;x, x^*} 
\str\{M_t  u_t^{-1} \}\ \sqrt{\det g(x^1 , \overline{x})}\, dx^1 d\overline x\nonumber\\
\sim &  \int_{\partial M} \text{vol}_{\partial M} (d\overline x)\int_{0}^{\frac{h}{\sqrt{t}}}\frac{e^{-2u^2}}{(2 \pi)^{\frac{d}{2}} t^{\frac{d-1}{2}}} 
\me_{(t,u, \overline x)} \str\{M_t u_t^{-1} \} du +R_t.\nonumber
\end{align}
Here we have used \rf{volcomp} and made the change of variable $x = \sqrt t \, u$.  For simplicity we have used a simplified notation
$$\me_{(t,u, \overline x)} := \mathbb{E}_{t;(\sqrt{t}u, \overline{x}), (-\sqrt t, u \overline x)}.$$
The error term $R_t$ has the bound
$$\vert R_t \vert \le \frac C{t^{d/2-1}}\int_{\partial M} \text{vol}_{\partial M} (d\overline x)\int_{0}^{\frac{h}{\sqrt{t}}} e^{- 2 u^2} 
\bigg\vert\me_{(t,u, \overline x)}\str\{M_t u_t^{-1} \}\bigg\vert du.$$

\begin{remark} The proof that the limit in \rf{integration1} exists and the error term tends to zero as $t\downarrow 0$ depends on the algebraic ``fantastic cancellation'' proved by 
Patodi~\cite{Patodi}.  For the case of a manifold without boundary, the cancellation occurs pathwise and the limit exists even before taking the expectation. 
For the case with boundary,  $M_t$ and $u_t$ contains the boundary local time and the limit exists only after taking the expectation.
\end{remark}

From the above computation, it is clear that all we need to show is the existence of an explicitly identifiable constant $C(u)$ such that 
$$\me_{(t,u, \overline x)} \str\{M_t u_t^{-1} \}\sim C(u) t^{(d-1)/2}.$$
This will imply that as $t\downarrow 0$ the error term vanishes and the limit of \rf{integration1} exists (after explicitly evaluating the integral of $e^{-2u^2}C(u)$ over $[0, \infty)$). 
This can be accomplished by studying the expansions of the stochastic parallel transport $u_t$ and the multiplicative functional $M_t$.  In the presence of boundary, this expansion
involves both time $t$ and the boundary local time $l_t$. 

By choosing an orthonormal frame at $x$,  we can regard  $u_t^{-1}$ as a (random) element taking values in the orthogonal group $O(d)$. We expect that when t is small, $u_t^{-1}$ 
is close to the identity map $I$.

\begin{lemma} \label{parallel}
Fix an $\overline x \in \partial M$. For any positive integer N, there is a constant $K_N$ such that
$$\mathbb{E}_{(t, u, x)} |u_t^{-1} - I|^N \leq K_N t^N\ .$$
\end{lemma}
\begin{proof} Let $x = (\sqrt tu, \overline x)$ for simplicity. Let $G(E)= |u_t^{-1} - I|^N$, then by definition $G(RE_t) = G(E_t)$. By Lemma \ref{reflection}, we have
$$\mathbb{E}_{t;x,x^*} \{G(E)\} = \frac{q(t,x,x )}{q(t,x,x^* )} \mathbb{E}_{t;x,x} \{  G(RE )\mathbb{I}_{T_{\partial M} \leq t}\}\ .$$
Since there exists a constant $C > 0$ such that $q(t,x,x )/q(t,x,x^* )\leq C$,  the rest of the proof follows that of {\sc Lemma} 7.7.4 of Hsu~\cite{Hsu3}.
\end{proof}	
This lemma guarantees that there is a unique $v_t \in so(d)$ such that 
$$u_t^{-1} = \exp v_t\quad\text{and}\quad  \mathbb{E}_{t;x,x^*} |v_t|^N \leq C_N t^N,$$
and $u_t^{-1}$ can be expanded as a power series of $v_t$. 

Now we come to the expansion of the multiplicative functional $M_t$. 
Iterating equation \rf{after}, we get the  an expansion for $M_t$ in the form 
$$M_t = \sum_{i=0}^lm_i(t) + R_l(t),\ \ \ \ i \in \frac{1}{2}\mathbb{N},$$
where $m_0 (t) = Q$ and
\begin{align}\label{iterate}
m_i (t) =  &\frac12\int_{0}^{t} m_{i - 1} (\tau)\Omega( E_\tau)Q\,d\tau - 
 \int_{0}^{t}m_{i - 1/2}(\tau)H(E_\tau)Q \, dl_{\tau} \\
\qquad\qquad &+\frac12\int_{0}^{t} \mathbb{I}_{\{\tau \ge T_{\partial M}\}}m_{i - 1}(\tau) \Omega( E_\tau) P\, d\tau.\nonumber
\end{align}
Each term in the epxansion is an iterated integral with respect to $ds$ and $dl_s$. More precisely, 
let 
\begin{equation} \label{decomposition1} m_i (t) =\displaystyle \sum_{p + q/2 = i} m_{p,q} (t)\ ,\end{equation}
where $m_{p,q}$ are the sum of terms in $m_i (t)$ such that $\Omega( u_t)dt $ appears $p$ times and $H(u_t)dl_t$ appears $q$ times. 
To find out the size of each term  in the expansion, we need to estimate the moments of the boundary local time. 

\begin{lemma} \label{bound}
For $x \in \partial M \times [0, h]$ and any positive integer $n$.  There is a constant $k_n$ such that
$$\mathbb{E}_{t;x,x^*} [l_t^n] \leq k_n t^{n/2},\ \ t \in [0,1].$$
\end{lemma}
\begin{proof} By symmetry we only need to consider the half interval $[0,t/2]$. Then
$$\mathbb{E}_{t;x,x} [l_t^n] = \frac{2\, \mathbb{E}_x \{q(t/2,E_{t/2},x)l_{t/2}^n\} }{q(t,x,x)}\ .$$
Note that there is a constant $C$ such that
$$q(t/2, E_{t/2},x ) \leq \frac{C}{t^{d/2}}\quad\text{and}\quad q(t,x,x) \ge \frac{C^{-1}}{t^{d/2}},$$
hence $\mathbb{E}_{t;x,x} [l_t^n] \leq 2 C^2 \mathbb{E}_{x} [l_{t/2}^n]$. 
By {\sc Lemma} 3.2 of Hsu~\cite{Hsu},  $\mathbb{E}_{x} [l_t^n] \leq C_n t^{n/2}$. The desired estimate follows.
\end{proof}

Applying {\sc Lemma} \ref{bound}, we have
\begin{equation} \label{multi}
\mathbb{E}_{t;x,x^*} \vert m_{p, q}(t)\vert \le C_{p, q} t^{p+q/2}
\end{equation}
and
$$\mathbb{E}_{t;x,x^*} |R_l(t)| \leq C_l t^{l+ 1/2}.$$

For the expansion of  $u_t^{-1} = \exp v_t$, we have
$$u_t^{-1} = \sum_{i=0}^{l'} \frac{(v_t)^i}{i!} + R_{l'}(t)$$
with the estimate
\begin{equation} \label{estimate2}
\mathbb{E}_{t;x,x^*} |R_{l'}(t)| \leq t^{l'+ 1}.
\end{equation}
by {\sc Lemma} \ref{parallel}. 
Now we have 
\begin{align}
&\me_{(t,u, \overline x)}  \str\{M_t  u_t^{-1} \} \label{products}\\
=&\me_{(t,u, \overline x)}  \str\lbr \left[\sum_{i=0}^{i=l}m_i (t) + R_l(t)\right]  u_t^{-1}\rbr\nonumber\\
=  &\me_{(t,u, \overline x)} \str\lbr\left[\sum_{p + q/2\le l}m_{p,q} (t) + R_l(t)\right]\left[\sum_{i=0}^{l'} \frac{(v_t)^i}{i!} + R_{l'}(t)\right]\rbr.\nonumber
\end{align}
The integrand of each $m_{p,q} (t)$ is a product of $\Omega(u_t)$, $H(u_t)$, $P$ and $Q$ and various time points. 
The cancellation of the lower order terms is the consequence of the following purely algebraic result due to Patodi~\cite{Patodi} 

\begin{lemma}\label{patodi} Suppose that $T_i\in \text{\rm End}(V)$. Extend each $T_i$ to be an element in $\text{\rm End}(\wedge^* V)$ and consider the composition 
$ T_1 \cdots T_l \in \text{\rm End}(\wedge^*V)$. If  if $l <\text{\rm dim} V$, then $\str (T_1 T_2 \ldots T_l) = 0$. If $l=\text{\rm dim} V$, then
$$\str(T_1 \cdots T_l) =  (-1)^l\cdot\text{coefficient of } x_1\cdots x_l\text{ in } \det\left(\sum_{i=1}^lx_i T_i\right).$$
In particular, $\str(T_1 \cdots T_l)$ is indpendent of the order of the factors $T_i$. 
\end{lemma}
\begin{proof} See {\sc Lemma} 7.3.2 in Hsu~\cite{Hsu}.  
\end{proof}

From this lemma we have
$$ \str\{m_{p,q}(t) (v_t)^i\} = 0 , \  \ \text{if}\  \ 2p + q + i < d - 1.$$
On the other hand, by \rf{multi} and \rf{estimate2}
$$\mathbb{E}_{t;x,x^*} |m_{p,q}(t) (v_t)^i| \leq Ct^{p+q/2+ i},\  \ \ \ x \in \partial M \times [0, h].$$
From \rf{integration1} and \rf{products} a typical term has the form
$$t^{-(d-1)/2}\me_{(t,u, \overline x)}\str\lbr m_{p,q}(t)(v_t)^i\rbr.$$
It algebraically vanishes when $2p+q+i< d-1$ and asymptotically vanishes when $2p+q+2i> d-1$. Hence the only terms we need to examine satisfy $2p+q+i\ge d-1$ and 
$2p+q+2i\le d-1$,  namely $i = 0$ and $2p+q = d-1$.  The number of such terms is $d/2$ if $d$ is even and $(d+1)/2$ if $d$ is odd.  If we can identify the the limits
$$\lim_{t\rightarrow0}\frac2{(2\pi)^{d/2}t^{(d-1)/2}}\int_0^\infty e^{-2u^2}\left[\me_{(t,u, \overline x)}\str\lbr m_{p,q}(t)\rbr\right]\, du  = C_{p, q}(\overline x), $$
then we will have
$$\chi(M) = \int_M e_M(x)\, dx + \int_{\partial M}e_{\partial M}(\overline x)\, d\overline x,$$
where 
$$e_{\partial M}(\overline x)= \sum_{2p+q = d-1}C_{p, q}(\overline x). $$

Since $m_0 = Q$ and the supertrace is invariant under permutations of the factors, we can drop the term involving the projection $P$  in the iteration formula \rf{iterate} 
because we can always move $P$ to be next to $Q$ without affecting the supertrace and $PQ = 0$. From now on we assume that $m_i$ has been redefined without the terms involving $P$. 
Thus after dropping inessential terms we are dealing with 
$$m(t) = \sum_{2p+q = d-1} \frac{(-1)^q}{2^p} \sum_{I\subset I_{p+q}}m_I(t),$$
where the sum is over all subsets of $I_{p+q} = \lbr 1, \cdots, p+q\rbr$ if size $p$, and 
\begin{align*}
m_I(t) =&\  \int_0^t d\eta_{p+q}\cdots\int_0^{\tau_{q+1}}d\eta_q
\int_0^{\tau_q}d\eta_{q-1}\cdots\int_0^{\tau_2}d\eta_1 \\
&\qquad \Omega(E_{\tau_{p+q}})\cdots \Omega(E_{\tau_{1+q}})H(E_{\tau_{q}})\cdots (E_{\tau_{1}}) Q,
\end{align*}
where $d\eta_i = d\tau_i$ if $i\in I$ and $d\eta_i = dl_{\tau_i}$ otherwise. 
Therefore it is enought to calculate the limit of 
$$t^{-(d-1)/2}\me_{(t, u, \overline x)}\left[\sum_{I\subset I_{p+q}}\str m_{\pi}(t)\right]$$
as $t\downarrow\infty$. 
Before doing so, we can make one more simplification. We note that as $t\downarrow 0$, the Brownian bridge shrink to the boundary point 
$\overline x$. If we replace the integrand in the integral by its value at $\overline x$,  the error is bounded by a  constant times 
$$\me_{(t, u, \overline x)}\int_0^t\cdots\int_0^t d(E_{\tau_i}, \overline x)d\tau_1\cdots d\tau_p\, dl_{p+1}\cdots dl_{\tau_{p+q}} \le C t^{(d-1)/2 + 1/2}.$$
This will create an error term that vanishes as $t\downarrow 0$.  After we substitute the integrand by its value at the limiting point $\overline x$ we are left with the sum of iterated integrals
$$\sum_{I\subset I_{p+q}} \int_0^t d\eta_{p+q}\cdots\int_0^{\tau_{q+1}}d\eta_q
\int_0^{\tau_q}d\eta_{q-1}\cdots\int_0^{\tau_2}d\eta_1 = \frac{ t^p l_t^q}{p!q!}.$$
Note that there are $(p+q)!/p!q!$ terms and all the iterated integrals are equal because they are the volume $t^pl_t^q/(p+q)!$ of the rectangular simplex with $p$ sides of length $t$ and $q$ 
sides of length $l_t$. 

It remains to evaluate the limit of  $t^{-q/2}\me_{(t, u, \overline x)}\left[l^q\right]$. 

\section{Moments of the boundary local time}  

We work in the semi-geodesic coordinates.  Denote the Brownian bridge from $(\sqrt{t}u,\overline{x})$ to $ (-\sqrt t u, \overline x)$ in 
time $t$ by $X^t$, whose normal component is $X^{1,t}$.  Scaling the time parameter by a factor of $t$ and let
$$Z^t_s = \frac{X^t_{st}}{\sqrt{t}}, \ 0\le s\le 1\quad\text{and}\quad l^t_1 = \frac{l_t}{\sqrt t}.$$
Then $l_1^t$ is the local time of the normal component $Z^{1,t}$ at $0$ up to time 1 and $Z^{1,t}_0= u$.   It is easy to write down the stochastic differential equation of $N^t = Z^{1,t}$ under the probability $\me_{(t, u, \overline x)}$.  We have
\begin{equation} \label{sde}
d N^t_s = d W_s + \sqrt{t}b(\sqrt{t}Z^t_s)ds + \sqrt{t}c(t(1-s),\sqrt{t}Z^t_s)ds, 
\end{equation}
where $W$ is a standard one-dimensional Brownian motion, $b^1$ is defined as in \rf{driftb} and 
$$c(s,x) = g^{1j}(z) \frac{\partial}{\partial x^j} \ln q(t,x,(u,\overline{x})).$$
From the explicit expression for $b^1$ it is easy to see that $\lim_{t\downarrow0}\sqrt t b(\sqrt t z) = 0$.  On the other hand,  we have (see {\sc Theorem} 2.5 of Norris~\cite{norris})
$$\lim_{t\downarrow0}\sqrt tc(t(1-s), \sqrt t z)=-2us+ \frac{ z^1+ u}{1-s}.$$
It follows that the normal component $Z^{1, t}$ converges to the process determined by the stochastic differential equation
$$dN_s =dW_s + \frac{N_s +u}{1-s}\, ds, \qquad N_0 = u.$$
This is the equation of a Brownian bridge from $u$ to $-u$.  Thus we have shown that
$$\lim_{t\rightarrow0}\me_{(t, u, \overline x)}\left[\left(\frac{l_t}{\sqrt t}\right)^q\right] = \me_{1, u, - u}\left[ l_1^q\right],$$
where $\me_{1, u, -u}$ is the expectation with respect to a Brownian bridge from $u$ to $-u$ in time 1 and $l_1$ is its local time at $0$ up to time 1.  The expectation can be computed explicitly.
According to {\sc Lemma} 4.2 of Shigekawa et al.~\cite{ShigekawaUW}, 
\begin{equation} \label{expectation2}
\frac1{\sqrt{2\pi}}\int_0^{\infty}e^{-2u^2} \mathbb{E}_{1, u, -u}\left[l_1^q\right]\, du = \frac{q!}{2^{q/2 + 2}\Gamma (q/2+1)}. 
\end{equation}

\section{Geometric identification of the boundary integrand}

Combining the calculations of the previous paragraphs, we conclude that the boundary integrand of the Gauss-Bonnet-Chern formula is given by
$$e_{\partial M} = \frac1{2(4\pi)^{(d-1)/2}}\sum_{2p+q = d-1}\frac{(-1)^q}{p! \varGamma(q/2+1)}\str(\Omega^pH^qQ).$$
Recall that $Q$ is the tangential projection at the boundary and $\str$ is the supertrace on $\wedge^*M$. Using the Gauss-Codazzi equation we can write 
the curvature tensor $\Omega$ of the manifold restricted to the boundary in terms of the curvature tensor of the boundary $\overline\Omega$ and the second fundamental form $H$.
This allows us to rewrite the boundary integrand $e_{\partial M}$ in terms of the geometry of the boundary $\partial M$ as an embedded submanifold of $M$.  More precisely we have
$\str [(D\Omega)^p (H)^q Q] = \overline{\str} [(D \Omega)^p (H)^q]$, 
where $\overline{\str}$ denotes the supertrace on $\bigwedge^*\partial M$. It follows that
$$e_{\partial M} = \frac1{2(4\pi)^{(d-1)/2}}\sum_{p = 0}^{[(d-1)/2]}\frac{(-1)^{q}}{p!\varGamma\left(q/2+1\right)}\overline{\str}(\overline\Omega^pH^q).$$
[$q = d-1-2p$]
\bigskip

{\sc Acknlowedgment.} The second author would like to thank Professor Denis Bell of the University of North Florida for his interest in the project and helpful discussions 
during the early stage of the investigation.

\end{document}